\documentclass[a4paper,11pt,twoside]{amsart}

\usepackage[latin1]{inputenc}
\usepackage{amsmath, amsfonts, amssymb,amsthm}
\usepackage[mathscr]{eucal} 
\usepackage[pdftex]{graphicx}
\usepackage{color}

\usepackage[T1]{fontenc}
\usepackage[sc]{mathpazo}
\usepackage[all]{xy}
\usepackage{hyperref}

\usepackage{enumerate}

\definecolor{alert}{rgb}{0.8,0,0}
\newcommand{\alert}[1]{\textbf{\textcolor{alert}{#1}}}

\newcommand{\alertm}[1]{%
	\marginpar{%
		\ifodd\value{page} \raggedright \else \raggedleft \fi
		\footnotesize{\alert{#1}}
	}
}

\newcommand{\N}{\mathbb{N}}

\newcommand{\R}{\mathbb{R}}
\newcommand{\C}{\mathbb{C}}
\newcommand{\s}{\mathbb{S}}
\newcommand{\h}{\mathbb{H}}
\newcommand{\E}{\mathbb{E}}
\newcommand{\M}{\mathbb{M}}
\newcommand{\Sl}{\mathrm{Sl}}

\newcommand{\X}{\mathfrak{X}}

\newcommand{\df}{\,\mathrm{d}}

\newcommand{\sm}{\smallsetminus}

\newcommand{\p}{\partial}

\DeclareMathOperator{\im}{Im}
\newcommand{\id}{\mathrm{id}}
\newcommand{\wt}{\widetilde}
\newcommand{\wh}{\widehat}
\renewcommand{\div}{\mathrm{div}}
\newcommand{\Ric}{\mathrm{Ric}}

\newtheorem{theorem}{Theorem}[section]
\newtheorem{proposition}[theorem]{Proposition}
\newtheorem{corollary}[theorem]{Corollary}
\newtheorem{lemma}[theorem]{Lemma}

\theoremstyle{definition}
  \newtheorem{definition}[theorem]{Definition}

\theoremstyle{remark}
\newtheorem{remark}[theorem]{Remark}

\numberwithin{equation}{section}
\title{On the classification of Killing submersions and their isometries}

\author{Jos\'{e} M. Manzano}
\address{Dpto. de Geometr\'{\i}a  y Topolog\'{\i}a \\
Universidad de Granada \\
18071 Granada, SPAIN} 
\email{jmmanzano@ugr.es}

\thanks{Research partially supported by the Spanish MCI research projects MTM2007-61775 and MTM2011-22547, and the Junta de Andaluc\'{i}a Grant P09-FQM-5088.}



\begin{document}

\begin{abstract}
A Killing submersion is a Riemannian submersion from an orientable $3$-manifold to an orientable surface whose fibers are the integral curves of a unit Killing vector field in the $3$-manifold. We classify all Killing submersions over simply-connected Riemannian surfaces and give explicit models for many Killing submersions including those over simply-connected constant Gaussian curvature surfaces. We also fully describe the isometries of the total space preserving the vertical direction. As a consequence, we prove that the only simply-connected homogeneous $3$-manifolds which admit a structure of Killing submersion are the $\E(\kappa,\tau)$-spaces, whose isometry group has dimension at least $4$.
\end{abstract}

\maketitle

\section{Introduction}

Simply-connected homogeneous Riemannian $3$-manifolds with isometry group of dimension $4$ or $6$ different from $\h^3$ can be represented by a $2$-parameter family $\E(\kappa,\tau)$, where $\kappa,\tau\in\R$. They include $\R^3$, $\s^3$, $\h^2\times\R$, $\s^2\times\R$, the Heisenberg group, the Berger spheres and the universal cover of the special linear group $\Sl_2(\R)$ endowed with a left-invariant metric, see~\cite{Daniel07,DHM09,mp11}. The $\E(\kappa,\tau)$-spaces are $3$-manifolds admitting a global unit Killing vector field whose integral curves are the fibers of a certain Riemannian submersion over the simply-connected constant Gaussian curvature surface $\M^2(\kappa)$. In the Riemannian product $3$-manifolds $M\times\R$, the projection over the first factor is a Riemannian submersion whose fibers are also the trajectories of a unit Killing vector field. In general, Riemannian submersions sharing this property will be called Killing submersions (see~\cite{EO,RST} and Definition~\ref{defi} below).

Constant mean curvature surfaces in $\E(\kappa,\tau)$ and $M\times\R$ have been extensively studied during the last decade and many results have been extended to the Killing submersion setting very recently (see~\cite{DL07,DL08,EO,LR,RST} for instance). Nevertheless, apart from aforementioned spaces, the theory of Killing submersions suffers from a lack of examples. It is necessary to mention that these $3$-manifolds are well-understood at the level of Differential Topology (see~\cite{Bes,GHV, Ste}) since the projection defines principal bundles with totally geodesic fibers. Nevertheless, the objective of this paper is to classify them in the Riemannian category provided that the base is simply-connected, and give explicit models depending on the base surface and a special geometric function, the so-called \emph{bundle curvature}.

Let $\pi:\E\to M$ be a differentiable submersion from a Riemannian $3$-manifold $\E$ onto a surface $M$. A vector $v\in T\E$ will be called \emph{vertical} when $v\in\ker(d\pi)$ and \emph{horizontal} when $v\in\ker(d\pi)^\bot$. The submersion $\pi$ is Riemannian when it preserves the length of horizontal vectors.

\begin{definition}\label{defi}
The Riemannian submersion $\pi:\E\to M$, where $\E$ and $M$ are connected and orientable, is called a Killing submersion if it admits a complete vertical unit Killing vector field.
\end{definition}

The definition is not as restrictive as it may seem, for all $3$-manifolds admitting a unit Killing vector field $\xi$ can be locally endowed of a Killing submersion structure. Namely, we can induce in a surface transverse to $\xi$ the metric of the distribution orthogonal to $\xi$. Since each each integral curve of $\xi$ locally intersects only once the transverse surface, this intersection locally defines a Killing submersion over the surface.

The bundle curvature of a Killing submersion $\pi:\E\to M$ is defined (see Lemma~\ref{lema:curvaturafibrado}) as the unique function $\tau\in \mathcal{C}^\infty(\E)$ satisfying
\[\overline\nabla_X\xi=\tau X\wedge\xi,\quad\text{for all }X\in\X(\E),\]
where $\xi$ is a vertical unit Killing vector field in $\E$. The bundle curvature is constant along the fibers of $\pi$ so it can be seen as a function $\tau\in \mathcal{C}^\infty(M)$ (in Propositions~\ref{prop:holonomia} and~\ref{prop:ejemplos_s2xR}, more geometric interpretations of the bundle curvature will be shown). This gives rise to some natural questions: Given a Riemannian surface $M$ and $\tau\in \mathcal{C}^\infty(M)$, does there exist a Killing submersion over $M$ with bundle curvature $\tau$? Is it unique? The main objective in Sections~\ref{sec:riemannian-properties} and~\ref{sec:existence} will be to give affirmative answer to these questions when $M$ is simply-connected. More specifically, we will classify Killing submersions up to isomorphism, in the following sense:

\begin{definition}\label{defi:isomorphism}
Let $\pi:\E\to M$ and $\pi':\E'\to M'$ be two Killing submersions. A (local) isomorphism of Killing submersions from $\pi$ to $\pi'$ is a pair $(f,h)$, where $h:M\to M'$ is an isometry and $f:\E\to\E'$ is a (local) isometry, such that $\pi'\circ f=h\circ\pi$. 
\end{definition}

Note that if $(f,h)$ is an isomorphism of Killing submersions, then $f$ maps fibers of $\pi$ into fibers of $\pi'$, and, if we consider a unit vertical Killing vector field $\xi$ in $\E$, then $f_*\xi$ is also a unit vertical Killing vector field in $\E'$.

Given a simply-connected Riemannian surface $M$ and $\tau\in \mathcal{C}^\infty(M)$, we will show that there exists a Killing submersion over $M$ with bundle curvature $\tau$, and it is unique (up to isomorphism) if the total space $\E$ is also simply-connected. In the process, it will turn out that the bundle curvature determines locally the geometry of the submersion, but the topology of $\E$ is also conditioned by the bundle curvature. More explicitly:
\begin{itemize}
 \item If $M$ is a topological disk, then the submersion is isomorphic to the projection $\pi_1:M\times\R\to M$, $\pi_1(p,t)=p$, for some Riemannian metric on $M\times\R$ such that $\partial_t$ is a unit vertical Killing vector field. In particular, the fibers of the submersion have infinite length.
 \item If, on the contrary, $M=(\s^2,g)$ for some Riemannian metric $g$, then we shall distinguish cases depending on whether the \emph{total bundle curvature} $T=\int_M\tau$ vanishes or not:
\begin{itemize}
 \item If $T=0$, then $\pi$ is isomorphic to $\pi_1:\s^2\times\R\to(\s^2,g)$, $\pi_1(p,t)=p$, for some metric on $\s^2\times\R$ such that $\partial_t$ is a unit vertical Killing vector field, so the fibers have infinite length.
 \item If $T\neq 0$, then $\pi$ is isomorphic to $\pi_{\text{Hopf}}:\s^3\to(\s^2,g)$ given by $\pi_{\text{Hopf}}(z,w)=\left(2z\bar w,(|z|^2-|w|^2\right)$,
where $\s^3\subset\C^2$ is endowed with a metric such that $\frac{\pi}{T}(iz,iw)$ is a unit vertical Killing vector field. In this case, the fibers have length $|2T|$.
\end{itemize}
\end{itemize}
When the total space is not simply-connected, Killing submersions over $M$ are also classified as the quotients of those listed above under a vertical translation (i.e., an element of the $1$-parameter group of isometries associated to the unit Killing vector field).

Though this theoretical description is exhaustive, we will give explicit models for a wide class of Killing submersions. Firstly, those over a disk with a conformal metric in terms of the conformal factor; the obtained examples will generalize the metrics for the $\E(\kappa,\tau)$-spaces in~\cite{Daniel07}. Secondly, we will obtain a general method to produce trivial Killing submersions (i.e., admitting a global smooth section) over any surface by isometrically embedding it in $\R^n$ for some $n\geq 3$. Finally, explicit models will also be obtained for Killing submersions over the round sphere $\s^2(\kappa)$ via the Hopf fibration (generalizing the metrics of the Berger spheres in~\cite{Tor12}).

The geometries of $M$ and $\E$ of a Killing submersion $\pi:\E\to M$ are well-related, and geodesics or isometries are good samples of that. On the one hand, geodesics of $\E$ can be divided into three different types: vertical ones, horizontal ones (which are horizontal lifts of geodesics of $M$) and those which are neither vertical nor horizontal, each of which makes a constant angle with the vertical direction and whose projection is well-understood (see Proposition~\ref{prop:killing-geodesicas}). In particular, $M$ is complete if and only if $\E$ is complete. On the other hand, a beautiful classification result is obtained when we look for what we will call \emph{Killing isometries} (i.e., isometries of $\E$ preserving the vertical direction).

Let $\pi:\E\to M$ be a Killing submersion, $\E$ and $M$ simply-connected, with bundle curvature $\tau\in \mathcal{C}^\infty(M)$. Then:
\begin{itemize}
 \item[a)] Given a Killing isometry $f:\E\to\E$, there exists a unique isometry $h:M\to M$ such that $\pi\circ f=h\circ\pi$. Moreover, $\tau\circ h=\tau$ if $f$ is orientation-preserving and $\tau\circ h=-\tau$ if it is orientation-reversing.
 \item[b)] Reciprocally, given an isometry $h:M\to M$ and $p_0,q_0\in\E$ satisfying $h(\pi(p_0))=\pi(q_0)$, the following properties hold:
 \begin{itemize}
  \item If $h\circ\tau=\tau$, then there is a unique orientation-preser\-ving Killing isometry $f:\E\to\E$ with $\pi\circ f=h\circ\pi$ and $f(p_0)=q_0$.
  \item If $h\circ\tau=-\tau$, then there is a unique orientation-reversing Killing isometry $f:\E\to\E$ with $\pi\circ f=h\circ\pi$ and $f(p_0)=q_0$.
 \end{itemize}
\end{itemize}

This construction provides a surjective group morphism from the group of Killing isometries of $\E$ to the group of isometries of $M$ which either preserve $\tau$ or map it to $-\tau$. The kernel of this morphism is the subgroup of isometries of $\E$ that leave the fibers invariant (i.e., the kernel consists of vertical translations, but also contains the symmetries with respect to a horizontal slice when $\tau=0$). In particular, $1$-parameter groups of isometries of $M$ preserving $\tau$ give rise to $1$-parameter groups of isometries in $\E$. Such groups have proven to be essential in surface theory, since they give rise to many geometric features, e.g., holomorphic quadratic differentials (see~\cite{AR}) and conjugate constructions (see~\cite{ManTor}) are related to them.

Finally, note that simply-connected homogeneous $3$-manifolds are classified: they are all isometric to Lie groups endowed with left-invariant metrics except for $\s^2(\kappa)\times\R$, where $\kappa>0$ (see~\cite[Theorem~2.4]{mp11}). 
In section~\ref{sec:homogeneos}, we will characterize the homogeneous spaces $\E(\kappa,\tau)$ as the only simply-connected $3$-dimensional homogeneous spaces admitting a Killing submersion structure (see Theorem~\ref{thm:killing-homogeneos}). Hence, the only Killing submersions whose total space is isometric to a Lie group endowed with a left invariant metric are the $\E(\kappa,\tau)$-spaces, except for $\s^2(\kappa)\times\R$.

\noindent\textbf{Acknowledgement.} I would like to express my gratitude to Hojoo Lee, Pablo Mira and Joaquín Pérez for suggesting some improvements in the preparation of this paper. I would also like to thank Luis Guijarro and Carlos Ivorra for showing me some useful references.

\section{Uniqueness results}\label{sec:riemannian-properties}

\subsection{The bundle curvature} 

We will begin by defining the bundle curvature. The following result can be found in~\cite[Proposition 2.6]{EO} but we will include the proof here for completeness.

\begin{lemma}\label{lema:curvaturafibrado}
Let $\pi:\E\to M$ be a Killing submersion. Then, there exists a function $\tau\in \mathcal{C}^\infty(\E)$ such that $\overline\nabla_X\xi=\tau X\wedge\xi$ for all $X\in\X(\E)$.

The function $\tau$ will be called the \emph{bundle curvature} of the submersion.
\end{lemma}

\begin{proof}
First of all, note that $\overline\nabla_\xi\xi=0$. Indeed, given $X\in\X(\E)$, we have $\langle\overline\nabla_\xi\xi,X\rangle=-\langle\overline\nabla_X\xi,\xi\rangle=-\frac{1}{2}X\langle\xi,\xi\rangle=0$ since $\xi$ is Killing and unitary.

Let us now take $X\in\X(\E)$ linearly independent with $\xi$. On the one hand, it is clear that $\langle\overline\nabla_X\xi,\xi\rangle=0$ and, on the other hand, $\langle\overline\nabla_X\xi,X\rangle=0$ since $\xi$ is Killing. Then, there exists a unique function  $\tau_X\in \mathcal{C}^\infty(\E)$ such that $\overline\nabla_X\xi=\tau_XX\wedge\xi$, so it suffices to prove that $\tau_X$ does not depend on $X$. It is clear that $\tau_X$ only depends on the horizontal part of $X$ so it will be enough to prove  that $\tau_X=\tau_Y$ for all $X,Y\in\X(\E)$ horizontal. By using again that $\xi$ is a Killing vector field, we get
\[\tau_Y\langle Y\wedge \xi,X\rangle=\langle\overline\nabla_Y\xi,X\rangle=-\langle \overline\nabla_X\xi,Y\rangle=-\tau_X\langle X\wedge\xi,Y\rangle=\tau_X\langle Y\wedge \xi,X\rangle,\]
so $\tau_X=\tau_Y$ where $X$ and $Y$ are linearly independent. In the rest of points, the identity $\tau_X=\tau_Y$ follows from the linearity of the connection.
\end{proof}

Observe that the function $\tau$ in the conditions of Lemma~\ref{lema:curvaturafibrado} is unique and its sign depends on the choice of orientation in $\E$. We will give now some consequences of this result in order to fix some notation.

\begin{remark}\label{rmk:tau}~
\begin{enumerate}
 \item The condition  $\overline\nabla_\xi\xi=0$ tells us that the fibers of the submersion are geodesics of $\E$, which will be called \emph{vertical geodesics}.
 \item The elements of the $1$-parameter group of isometries $\{\phi_t\}_{t\in\R}$ associated to the Killing vector field $\xi$ will be called \emph{vertical translations}. 

Note that $\phi_t$ preserves the Killing field $\xi$ and the orientation in $\E$. Thus, if we apply $d\phi_t$ to the identity in Lemma~\ref{lema:curvaturafibrado}, we easily get $\tau=\tau\circ\phi_t$ for all $t\in\R$. This means that the bundle curvature is constant along the fibers and, hence, it may be consider as a function either in $\E$ or in the base $M$.
 \item More generally, let $(f,h)$ an isomorphism between two Killing submersions $\pi:\E\to M$ and $\pi':\E'\to M'$ (see Definition~\ref{defi:isomorphism}) and define $\tau\in \mathcal{C}^\infty(\E)$ and $\tau'\in \mathcal{C}^\infty(\E')$ as their bundle curvatures with respect to some orientations in $\E$ and $\E'$, respectively. Then, $\tau\circ f=\tau$ when $f$ preserves the orientation, or $\tau\circ f=-\tau'$ when $f$ reverses the orientation.
\end{enumerate}
\end{remark}

In the product spaces $M\times\R$ the projection over the first factor is a Killing submersion, so its bundle curvature is $\tau\equiv 0$ (from Lemma~\ref{lema:curvaturafibrado} it is easy to deduce that $\tau\equiv 0$ in a Killing submersion if and only if the horizontal distribution in the total space is integrable). Given $\kappa,\tau\in\R$, there exists a Killing submersion $\pi:\E(\kappa,\tau)\to\M^2(\kappa)$ with constant bundle curvature $\tau$. If $\kappa>0$ and $\tau\neq 0$, the projection is the Hopf fibration and we obtain the Berger spheres; in the rest of cases the fibers have infinite length. We refer the reader to~\cite{Daniel07} for a description of these examples, although Berger spheres from a global point of view can be found in~\cite{Tor12}.

Other examples derived from the aforementioned ones are their Riemannian quotients by a convenient vertical translation. Thus the length of the fibers will play an important role in the theory. Since fibers are geodesics, the following result follows from~\cite[Theorem 9.56]{Bes}.

\begin{lemma}\label{lema:longitud-fibras}
Let $\pi:\E\to M$ be a Killing submersion. Then all the fibers of $\pi$ share the same (finite or infinite) length.
\end{lemma}


%

\subsection{Local representation of a Killing submersion}\label{subsec:killing-estructura} Given a surface $M$ and $\tau\in \mathcal{C}^\infty(M)$, we are interested in finding all Killing submersions over $M$ with bundle curvature $\tau$. Let us begin by giving a useful technical tool that will simplify some reasonings along this paper.

\begin{proposition}\label{prop:existencia-seccion}
Let $\pi:\E\to M$ be a Killing submersion, and suppose that $M$ is noncompact. Then, $\pi$ admits a global smooth section $F:M\to\E$. Hence,
\[\Psi:M\times\R\to\E,\quad \Psi(p,t)=\phi_t(F(p)),\]
is a local diffeomorphism, where $\{\phi_t\}$ denotes the $1$-parameter group of vertical translations. Moreover, $\Psi$ is a global diffeomorphism if and only if the fibers of $\pi$ have infinite length.
\end{proposition}

\begin{proof}
We can suppose that the fibers of $\pi$ have finite length (otherwise, we take a quotient of $\pi$ under a vertical translation $\phi_{t}$ for some $t>0$). Then, $\pi$ is a codimension one circle bundle over a noncompact surface and~\cite[Section VII.5]{GHV} yields the existence of a global smooth section. Moreover, $\Psi$ is a local diffeomorphism since its differential is injective at every point.

Finally, note that $\Psi$ is a global diffeomorphism if and only if it is injective, but $\Psi(p',t')=\Psi(p,t)$ implies $p=p'$ since $\Psi(p',t')$ and $\Psi(p,t)$ belong to the same fiber of $\pi$, so last assertion in the statement holds.
\end{proof}

This result will be mostly used to ensure that there exists a smooth section $F:U\to\E$ for any coordinate chart $(U,\varphi)$ in $M$, but it also implies that exceptional topologies for the total space may only arise when the base is compact. Note that, if the base is compact, then Proposition~\ref{prop:existencia-seccion} no longer holds as the Hopf fibration from $\s^3$ to $\s^2$ shows. 

The following result will be the cornerstone of the subsequent development yielding a standard way of describing $\pi$ in terms of  $M$ and $\tau$.

\begin{proposition}\label{prop:trivializacion}
Let $\pi:\E\to M$ be a Killing submersion. Let $U\subset M$ be an open set such that there is a conformal diffeomorphism $\varphi:U\to\Omega\subset\R^2$. Then:
\begin{itemize}
 \item[a)] Given a smooth section $F_0:U\to\pi^{-1}(U)$, the transformation
\begin{equation}\label{eqn:isometria-local-killing}
\begin{array}{rcl}
f:\Omega\times\R&\longrightarrow&\pi^{-1}(U)\\
(x,y,t)&\longmapsto&\phi_t(F_0(\varphi^{-1}(x,y))
\end{array}
\end{equation}
is a local diffeomorphism and satisfies $\pi\circ f=\varphi\circ\pi_1$ in $\Omega\times\R$, where $\pi_1:\Omega\times\R\to\Omega$ is the proyection over the first factor.
 \item[b)] Let us write the induced metric in $\Omega$ as $\df s_\lambda^2=\lambda^2(\df x^2+\df y^2)$ for some $\lambda\in\mathcal{C}^\infty(\Omega)$ positive. Then, there exist $a,b\in\mathcal{C}^\infty(\Omega)$ such that the metric in $\Omega\times\R$ which makes $f$ a local isometry can be expressed as
\begin{equation}\label{eqn:metrica-general-killing}
\df s^2=\lambda^2(\df x^2+\df y^2)+(\df t-\lambda(a\df x+b\df y))^2.
\end{equation}
\item[c)]  $\pi_1:(\Omega\times\R,\df s^2)\to (\Omega,\df s_\lambda^2)$ is a Killing submersion with unit Killing vector field $\partial_t$, and $(f,\varphi^{-1})$ is a local isomorphism from $\pi_1$ to $\pi$.
\end{itemize}
Moreover, if the fibers of $\pi$ have infinite length, then $f$ is a global diffeomorphism.
\end{proposition}

\begin{proof}
We deduce from Proposition~\ref{prop:existencia-seccion} that $\Psi:U\times\R\to\pi^{-1}(U)$ given by $\Psi(p,t)=\phi_t(F_0(p))$ is a local diffeomorphism, so $f=\Psi\circ(\varphi^{-1}\times\id_\R)$ is also a local diffeomorphism, and it obviously satisfies the condition $\varphi\circ\pi_1=\pi\circ f$, so item (a) is proved. Note that Proposition~\ref{prop:existencia-seccion} also ensures that $f$ is a global diffeomorphism if the fibers of $\pi$ have infinite length.

To prove item (b), consider the unique Riemannian metric $\df s^2$ in $\varphi(U)\times\R$ making $f$ a local isometry. The condition $\varphi\circ\pi_1=\pi\circ f$ implies that $\pi_1$ is a Killing submersion. Vertical translations for $\pi$ correspond (through $f$) to isometries of the form $(x,y,t)\to(x,y,t+\mu)$, $\mu\in\R$, in $(\varphi(U)\times\R,\df s^2)$. In particular, $\p_t$ is a unit vertical Killing vector field in $(\varphi(U)\times\R,\df s^2)$.

Let $\{e_1,e_2\}$ be the orthonormal frame in $(\varphi(U),\df s_\lambda^2)$, where $e_1=\frac{1}{\lambda}\p_x$ and $e_2=\frac{1}{\lambda}\p_y$, and let $\{E_1,E_2\}$ be the horizontal lift of $\{e_1,e_2\}$ with respect to $\pi_1$ and $E_3=\partial_t$. Since $\pi_1$ is the projection over the first two variables, there exist $a,b\in\mathcal{C}^\infty(\varphi(U))$ such that 
\begin{equation}\label{eqn:base-universal-killing}
\left\{\begin{array}{l}
(E_1)_{(x,y,t)}=\tfrac{1}{\lambda(x,y)}\partial_x+a(x,y)\partial_t,\\
(E_2)_{(x,y,t)}=\tfrac{1}{\lambda(x,y)}\partial_y+b(x,y)\partial_t,\\
(E_3)_{(x,y,t)}=\partial_t.
\end{array}\right.
\end{equation}
Note that $\{E_1,E_2,E_3\}$ is an orthonormal frame in $(\varphi(U)\times\R,\df s^2)$ which can be supposed positively oriented after possibly swapping $e_1$ and $e_2$. Now it is straightforward to show that the global frame~\eqref{eqn:base-universal-killing} is orthonormal for $\df s^2$ if and only if $\df s^2$ is the metric given by~\eqref{eqn:metrica-general-killing}.
\end{proof}

Regardless of the values of the functions $a,b\in\mathcal{C}^\infty(\Omega)$, the Riemannian metric given by equation~\eqref{eqn:metrica-general-killing} satisfies that the projection over the first two variables is a Killing submersion over $(\Omega,\df s_\lambda^2)$.

\begin{definition}[canonical example]\label{defi:canonical}
Given an open set $\Omega\subset\R^2$ and $\lambda,a,b\in\mathcal{C}^\infty(\Omega)$ with $\lambda>0$, the Killing submersion
\[\pi_1:(\Omega\times\R,\df s^2_{\lambda,a,b})\to (\Omega,\df s_\lambda^2),\quad\pi_1(x,y,z)=(x,y),\]
\begin{equation*}
\df s^2_{\lambda,a,b}=\lambda^2(\df x^2+\df y^2)+(\df z-\lambda(a\df x+b\df y))^2.
\end{equation*}
will be called the canonical example associated to $(\lambda,a,b)$.
\end{definition}

Equation~\eqref{eqn:base-universal-killing} defines a global orthonormal frame $\{E_1,E_2,E_3\}$ for $\df s^2_{\lambda,a,b}$, where $E_1$ and $E_2$ are horizontal, and $E_3$ is a unit vertical Killing field. It is easy to check that $[E_1,E_3]=[E_2,E_3]=0$ and
\[[E_1,E_2]=\tfrac{\lambda_y}{\lambda^2}E_1-\tfrac{\lambda_x}{\lambda^2}E_2+\left(\tfrac{1}{\lambda^2}(b\lambda_x-a\lambda_y)+\tfrac{1}{\lambda}(b_x-a_y)\right)E_3.\]
Taking into account Lemma~\ref{lema:curvaturafibrado}, we can compute the bundle curvature $\tau$ associated to this canonical example as
\begin{equation}\label{eqn:tau-divergencia-killing}
\begin{split}
2\tau&=\langle\overline\nabla_{E_1}{E_2},E_3\rangle-\langle\overline\nabla_{E_2}{E_1},E_3\rangle=\langle[E_1,E_2],E_3\rangle\\
&=\tfrac{1}{\lambda^2}(b\lambda_x-a\lambda_y)+\tfrac{1}{\lambda}(b_x-a_y)=\tfrac{1}{\lambda^2}\left((\lambda b)_x-(\lambda a)_y\right).
\end{split}
\end{equation}
This divergence formula will come in handy in the sequel.

\begin{lemma}[Classification of canonical examples]\label{lema:clasificacion-local-killing}
 Let $\Omega\subseteq\R^2$ be a simply-connected open set and $\lambda,a_0,a_1,b_0,b_1\in\mathcal{C}^\infty(\Omega)$ such that $\lambda>0$. The following assertions are equivalent:
\begin{itemize}
 \item[i)] There exists $d\in\mathcal{C}^\infty(\Omega)$ such that the pair $(f_d,\id_\Omega)$, where 
\begin{equation}\label{eqn:killing-isomorfismo-local}
\begin{array}{rcl}
f_d:(\Omega\times\R,\df s_{\lambda,a_0,b_0}^2)&\longrightarrow&(\Omega\times\R,\df s_{\lambda,a_1,b_1}^2)\\(x,y,z)&\longmapsto&(x,y,z-d(x,y)),
\end{array}
\end{equation}
is an isomormphism of Killing submersions.
 \item[ii)] There exists $d\in\mathcal{C}^\infty(\Omega)$ such that $d_x=\lambda(a_1-a_0)$ and $d_y=\lambda(b_1-b_0)$.
\item[iii)] The bundle curvatures $\tau_0,\tau_1\in\mathcal{C}^\infty(\Omega)$ of the two submersions coincide.
\end{itemize}
\end{lemma}

\begin{proof}
It is easy to check that $f_d$ is an isometry if and only if $d$ satisfies (ii) so the equivalence between (i) and (ii) is proved. Since $f_d$ preserves the orientation in $\Omega\times\R$,  we get that (i) implies (iii) from Remark~\ref{rmk:tau}. Finally, to prove that (iii) implies (ii), observe that $\tau_0=\tau_1$ means $(\lambda b_0)_x-(\lambda a_0)_y=(\lambda b_1)_x-(\lambda a_1)_y$ in view of~\eqref{eqn:tau-divergencia-killing}. Equivalently, we have $(\lambda(a_1-a_0))_y=(\lambda(b_1-b_0))_x$, so item (ii) follows from Poincaré's Lemma and the fact that $\Omega$ is simply connected.
\end{proof}

We remark that condition (i) in the statement is equivalent to the fact that the canonical examples for $(\lambda,a_0,b_0)$ and $(\lambda,a_1,b_1)$ represent the same Killing submersion for different initial sections. The function $d$ is, up to an additive constant, the vertical distance between such sections.

Lemma~\ref{lema:clasificacion-local-killing} is actually a local classification result for Killing submersions, since we have proved that all Killing submersions are locally equivalent to canonical examples. We will now give the general version.

\begin{theorem}[Uniqueness]\label{thm:unicidad}
For $i\in\{0,1\}$, let $\pi_i:\E_i\to M_i$ be a Killing submersion, $M_i$ being simply-connected, with bundle curvature $\tau_i\in\mathcal{C}^\infty(M_i)$ for a given orientation in $\E_i$. Suppose that the fibers of $\pi_0$ and the fibers of $\pi_1$ have the same lenght and there exists an isometry $h:M_0\to M_1$. Let $p_0\in\E_0$ and $p_1\in\E_1$ be such that $h(\pi_0(p_0))=\pi_1(p_1)$. 
\begin{itemize}
 \item[a)] If $\tau_1\circ h=\tau_0$, then there exists a unique orientation-preserving isometry $f:\E_0\to\E_1$ such that $\pi_1\circ f=h\circ\pi_0$ and $f(p_0)=p_1$.
 \item[b)] If $\tau_1\circ h=-\tau_0$, then there exists a unique orientation-reversing isometry $f:\E_0\to\E_1$ such that $\pi_1\circ f=h\circ\pi_0$ and $f(p_0)=p_1$.
\end{itemize} 
\end{theorem}

\begin{proof}
Let us first consider the case of $M_i$ being a topological disk, so there exist conformal diffeomorphisms $\varphi_i: M_i\to\Omega$ such that $h\circ\varphi_0=\varphi_1$, where $\Omega\subset\R^2$ is an open set. For $i\in\{0,1\}$, Proposition~\ref{prop:existencia-seccion} guarantees the existence of a global smooth section $F_i:M_i\to\E_i$ and a local diffeomorphism $f_i:\Omega\times\R\to\E_i$, given by $f_i(x,y,t)=\phi^i_t(F_i(\varphi_i^{-1}(x,y)))$ as in Proposition~\ref{prop:trivializacion}, where $\{\phi^i_t\}$ is the $1$-parameter group of vertical translations associated to $\pi_i$. In other words, we obtain a commutative diagram as in Figure~\ref{fig:killing-isometrias}, where $\pi:\Omega\times\R\to\Omega$ is the projection over the first factor.

\begin{figure}
\[
\xymatrix{\E_0\ar^{\pi_0}[rrrr]\ar@{..>}_{f}[ddd]&&&&M_0\ar^{h}[ddd]\ar_{\varphi_0}[ld]\\
&\Omega\times\R\ar_{f_0}[lu]\ar^{\pi}[rr]\ar@{..>}^{\widehat f}[d]&&\Omega\ar@{=}[d]&\\
&\Omega\times\R\ar^{f_1}[ld]\ar^{\pi}[rr]&&\Omega&\\
\E_1\ar_{\pi_1}[rrrr]&&&&M_1\ar^{\varphi_1}[lu]
}
\]
\caption{Horizontal and vertical arrows represent Killing submersions and isometries, respectively. Diagonal ones relate the original diagram with the canonical examples.}\label{fig:killing-isometrias}
\end{figure}
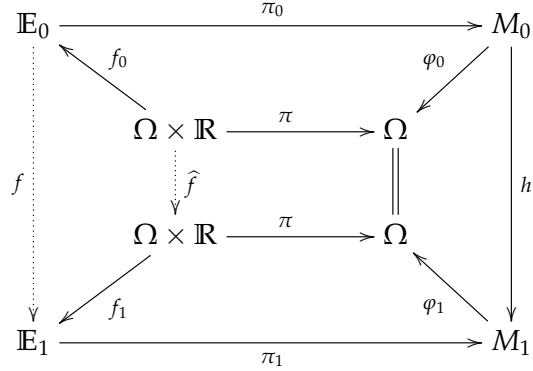

Observe that $\varphi_0$ and $\varphi_1$ induce the same metric $\lambda^2(\df x^2+\df y^2)$ on $\Omega$, and $f_0$ and $f_1$ induce canonical metrics $\df s^2_{\lambda,a_0,b_0}$ and $\df s^2_{\lambda,a_1,b_1}$ on $\Omega\times\R$, respectively. Moreover, the condition $\tau_1\circ h=\tau_0$ ensures that both canonical examples for $(\lambda,b_0,a_0)$ and $(\lambda,b_1,a_1)$ have the same bundle curvature so Lemma~\ref{lema:clasificacion-local-killing} yields the existence of a isometry 
\[\wh f:(\Omega\times\R,\df s^2_{\lambda,a_0,b_0})\to(\Omega\times\R,\df s^2_{\lambda,a_1,b_1})\]
of the form $\wh f(x,y,z)=f(x,y,z+d(x,y))$ for some $d\in\mathcal{C}^\infty(\Omega)$, so $\pi\circ\wh f=\pi$. If $\tau_1\circ h=-\tau_0$, the canonical examples for $(\lambda,b_0,a_0)$ and $(\lambda,b_1,a_1)$ have opposite bundle curvatures, so it is easy to see that there exists an isometry of the form $\wh{f}(x,y,z)=f(x,y,-z-d(x,y))$ for some $d\in\mathcal{C}^\infty(\Omega)$.

In both cases, the isometry $\wh{f}$ induces an isometry from the quotient of $(\Omega\times\R,\df s^2_{\lambda,a_0,b_0})$ by a vertical translation to the quotient of  $(\Omega\times\R,\df s^2_{\lambda,a_1,b_1})$ by the same vertical translation. Adjusting the translation so that the length of the fibers of the quotient is the same as in $\E_0$ or $\E_1$, the isometry in the quotient provides an isometry $f:\E_0\to\E_1$ such that $\pi_1\circ f=h\circ\pi_0$. We get $f(p_0)=p_1$ by just composing $f$ with a vertical translation.

Finally, suppose that $M_i$ are topological $2$-spheres. Let $U_0=M_0\sm\{q_0\}$ for some $q_0\neq\pi_0(p_0)$ and $U_1=h(U_0)=M_1\sm\{h(q_0)\}$. Note that $h:U_0\to U_1$ is an isometry in the conditions of the disk-case so it lifts to an isometry $f:V_0\to V_1$, where $V_i=\pi_i^{-1}(U_i)$, $i\in\{0,1\}$, satisfying $\pi_1\circ f=h\circ\pi_0$ in $V_0$ and $f(p_0)=p_1$. Now, let $\wt{p}_0\in V_0$ be such that $\pi_0(\wt p_0)\neq\pi_0(p_0)$, and $\wt p_1=f(\wt p_0)$. Let us take $\wt q_0\in M_0$ such that $\wt q_0\not\in\{\pi_0(\wt p_0),q_0\}$, and $\wt U_0=M_0\sm\{\wt q_0\}$, $\wt U_1=h(\wt U_0)=M_1\sm\{h(\wt q_0)\}$. The same reasoning above gives an isometry $\wt f:\wt V_0\to \wt V_1$, where $\wt V_i=\pi_i^{-1}(\wt U_i)$, $i\in\{1,2\}$, satisfying the condition $\pi_1\circ\wt f=h\circ\pi_0$ in $\wt V_0$ and $\wt f(\wt p_0)=\wt p_1=f(\wt p_0)$. 

Since $V=V_0\cap\wt V_0$ is connected, $f(\wt p_0)=\wt f(\wt p_0)$, and $(\df f)_{\wt p_0}=(\df\wt f)_{\wt p_0}$ (this is because both $f$ and $\wt f$ preserve the vertical direction and $\pi_1\circ\wt f=h\circ\pi_0=\pi_1\circ f$ in $V$), we conclude that $f=\wt f$ in $V$. As $V_0\cup\wt V_0=\E_0$, we deduce that $f$ can be extended (by $\wt f$) to an isometry from $\E_0$ to $\E_1$, and it trivially satisfies the conditions in the statement.
\end{proof}

\subsection{Killing isometries}
We will now particularize some of the results in the previous section to study isometries of the total space of a Killing submersion $\pi:\E\to M$ preserving the Killing submersion structure, i.e., those preserving the direction of a unit vertical Killing vector field $\xi$.

\begin{definition}
In the previous notation, the isometries of $\E$ satisfying $f_*\xi=\xi$ or $f_*\xi=-\xi$ will be called Killing isometries. 
\end{definition}

The definition does not depend on the choice of $\xi$. If $f_*\xi=\xi$ (resp. $f_*\xi=-\xi$), then $f$ is said to preserve (resp. reverse) the orientation of the fibers. Note that preserving the orientation of the fibers is not related to preserving or reversing the orientation of the total space $\E$.

%

\begin{lemma}\label{lema:killing-isometrias-proyeccion}
Let $\pi:\E\to M$ be a Killing submersion with bundle curvature $\tau\in \mathcal{C}^\infty(M)$, and let $f:\E\to\E$ be a Killing isometry. Then:
\begin{itemize}
 \item[a)] There exists a unique isometry $h:M\to M$ such that $\pi\circ f=h\circ\pi$.
 \item[b)] If $f$ preserves the orientation in $\E$, then $\tau\circ h=\tau$; 
 \item[c)] If $f$ reverses the orientation in $\E$, then $\tau\circ h=-\tau$.
\end{itemize}
\end{lemma}

\begin{proof}
Item (a) follows from the fact that $f$ maps fibers to fibers and from the fact that $\df\pi$ is an isometry when restricted to the horizontal distribution. Now, it is easy to see that $(f,h)$ is an isomorphism of Killing submersions 
(see Definition~\ref{defi:isomorphism}) so items (b) and (c) follow from Remark~\ref{rmk:tau}.
\end{proof}

In fact, the map $f\mapsto h$ defined by item (a) of Lemma~\ref{lema:killing-isometrias-proyeccion} can be easily proved to be a group morphism from the group of Killing isometries to the group of isometries of $M$ with $\tau\circ h=\pm \tau$. Moreover, the normal subgroup of orientation-preserving isometries is mapped to those isometries of $M$ which preserve $\tau$. As an application of Theorem~\ref{thm:unicidad}, we can prove that this morphism is surjective and its kernel consists of the vertical translations and,  for $\tau\equiv 0$, also the symmetries with respect to a slice.

\begin{corollary}\label{coro:killing-isometrias}
Let $\pi:\E\to M$ be a Killing submersion with bundle curvature $\tau\in \mathcal{C}^\infty(M)$ and suppose that $M$ is simply-connected. Let $h:M\to M$ be an isometry and take $p_0,q_0\in\E$ such that $h(\pi(p_0))=\pi(q_0)$.
\begin{itemize}
 \item[a)] If $\tau\circ h=\tau$ in $M$, then there exists a unique orientation-preserving Killing isometry $f:\E\to\E$ such that $\pi\circ f=h\circ\pi$ and $f(p_0)=q_0$.
 \item[b)] If $\tau\circ h=-\tau$ in $M$, then there exists a unique orientation-reversing Killing isometry $f:\E\to\E$ such that $\pi\circ f=h\circ\pi$ and $f(p_0)=q_0$.
\end{itemize}
\end{corollary}

As an immediate consequence, in the following two situations there do not exist Killing isometries reversing the orientation of the total space:
\begin{itemize}
 \item If the bundle curvature is a non-zero constant.
 \item If $M$ is a Riemannian $2$-sphere and $\int_{M}\tau\neq 0$.
\end{itemize}

\section{Curves in Killing submersions}\label{sec:curves}

\subsection{The horizontal lift of a curve}

\begin{definition}
Let $\pi:\E\to M$ be a Killing submersion and $\alpha:[c,d]\to M$ a $\mathcal{C}^1$-curve. A horizontal (or legendrian) lift of $\alpha$ is a $\mathcal{C}^1$-curve $\wt\alpha:[c,d]\to\E$  such that $\wt\alpha'$ is always horizontal and $\pi\circ\wt\alpha=\alpha$ in $[c,d]$. 
\end{definition}

This concept extends to piecewise $\mathcal{C}^1$-curves $\alpha:[c,d]\to M$, i.e., such that there is a partition $c=t_0<t_1<\ldots<t_n=d$ so that $\alpha_{|[t_{i-1},t_i]}$ is $\mathcal{C}^1$ for all $i\in\{1,\ldots,n\}$. A horizontal lift of $\alpha$ is a continuous curve $\wt\alpha:[c,d]\to\E$ such that $\wt\alpha_{|[t_{i-1},t_i]}$ is a horziontal lift of $\alpha_{|[t_{i-1},t_i]}$ for all $i\in\{1,\ldots,n\}$.

\begin{lemma}
Let $\alpha:[c,d]\to M$  be a piecewise $\mathcal{C}^1$-curve. Given $p_0\in \E$ such that $\pi(p_0)=\alpha(c)$, there exists a unique horizontal lift $\wt\alpha$ of $\alpha$ such that $\wt\alpha(c)=p_0$.
\end{lemma}

\begin{proof}
Let $c=t_0<t_1<\ldots<t_n=d$ be a partition such that $\alpha_{|[t_{i-1},t_i]}$ is a $\mathcal{C}^1$-curve. We can refine the partition so that $\alpha([t_{i-1},t_i])\subset U_i$ for some conformal chart $(U_i,\varphi_i)$ of $M$ for all $i$. Thus, we can assume that $\alpha$ is contained in such a chart $(U,\varphi)$, so $\wt\alpha$ will be contained in $\pi^{-1}(U)$. 

This allows us to work in the canonical example in Definition~\ref{defi:canonical} for $\Omega=\varphi(U)$ and some $\lambda,a,b\in \mathcal{C}^\infty(\varphi(U))$, $\lambda>0$. Writing in coordinates $\alpha(t)=(x(t),y(t))\in\varphi(U)$, a horizontal lift of $\alpha$ must be of the form $\wt\alpha(t)=(x(t),y(t),z(t))$ for some $z:[c,d]\to\R$, and must satisfy $\langle\wt\alpha',\p_z\rangle=0$. This last condition can be developed as
\begin{equation}\label{eqn:caracterizacion-legendriano}
z'=\lambda(x,y)\cdot\left(a(x,y)x'+b(x,y)y'\right).
\end{equation}
Since $\pi(p_0)=\alpha(c)$, we have $p_0=(x(c),y(c),z_0)$ for some $z_0\in\R$. We deduce that there exists a unique $\mathcal{C}^1$-function $z(t)$ satisfying~\eqref{eqn:caracterizacion-legendriano} with initial condition $z(c)=z_0$, so the horizontal lift exists and is unique.
\end{proof}

We can now give a geometric meaning of the bundle curvature in terms of the difference of heights of the endpoints of the horizontal lift of closed curves (see also~\cite[Proposition 1.6.2]{DHM09}). Supposing that the fibers have infinite length will be necessary for the difference of heigths to make sense.

\begin{proposition}\label{prop:holonomia}
Let $\pi:\E\to M$ be a Killing  submersión whose fibers have infinite length. Given a simple piecewise $\mathcal{C}^1$-curve $\alpha:[c,d]\to M$ bounding an orientable relatively compact open set $G\subset M$, and a horizontal lift $\wt\alpha$ of $\alpha$, then
\[\left|\int_G\tau\right|=\frac{h}{2},\]
where $h$ is the length of the vertical segment joining $\wt\alpha(c)$ and $\wt\alpha(d)$.
\end{proposition}

\begin{proof}
Let us consider an atlas of $M$  consisting of conformal charts. We will first suppose that $\overline G$ is contained in one of the charts $(U,\varphi)$, so we can suppose that we are working in the canonical example given by Definition~\ref{defi:canonical} for $\lambda,a,b\in \mathcal{C}^\infty(\varphi(U))$, $\lambda>0$. Moreover, equation~\eqref{eqn:tau-divergencia-killing} allows us to write $\tau$ as a divergence in $\varphi(G)$. The divergence theorem yields
\[\int_{\varphi(G)}2\tau=\int_{\varphi(G)}\div\left(\tfrac{b}{\lambda}\partial_x-\tfrac{a}{\lambda}\p_y\right)=\int_{\partial\varphi(G)}\langle\tfrac{b}{\lambda}\partial_x-\tfrac{a}{\lambda}\p_y,\eta\rangle,\]
where $\eta$ is the outer unit conormal to $\varphi(G)$ along its boundary. We write in coordinates  $\alpha=(x,y)$ and $\wt\alpha=(x,y,z)$, and suppose $\alpha$ is parametrized by arc-length (i.e., $(x')^2+(y')^2=\frac{1}{\lambda^2}$). Hence $\eta=-y'\p_x+x'\p_y$, up to a sign, so we deduce from equation~\eqref{eqn:caracterizacion-legendriano} that
\[\left|\int_G 2\tau\right|=\left|\int_c^d\lambda\cdot (ax'+by')\right|=\left|\int_c^d z'\right|=|z(d)-z(c)|.\] 
As $h=|z(d)-z(c)|$ in this model, we are done.

If $\overline G$ is not contained in a single chart,  we can triangulate the open set $\overline G$ in a finite number of triangles with piecewise $\mathcal{C}^1$-boundaries in such a way each triangle is contained in a coordinate chart of the atlas (see the proof of~\cite[Theorem 2.3.A.1]{Jost} where a similar triangulation is constructed) and $\alpha$  can be expressed as a finite sum of the boundaries of these triangles. As $G$ is orientable, such boundaries can be oriented so that the interior ones cancel out in pairs. The argument above applied to each triangle together with the divergence theorem gives the desired result.
\end{proof}

\subsection{Geodesics}
Let $\pi:\E\to M$ be a Killing submersion. Given two vector fields $X,Y\in\X(M)$, we can consider their horizontal lifts $\overline X,\overline Y\in\X(\E)$. Then, the following equality holds (see~\cite[pp.185-187]{doC}):
\begin{equation}\label{conexion-levantamiento}
\overline\nabla_{\overline X}{\overline Y}=\overline{\nabla_XY}+[\overline X,\overline Y]^v,
\end{equation}
where $\nabla$ and $\overline\nabla$ is the Levi-Civita connections in $M$ and $\E$, respectively, $\overline{\nabla_XY}$ is the horizontal lift of $\nabla_XY$ and $[\overline X,\overline Y]^v$ is the vertical part of $[\overline X,\overline Y]$.

From~\eqref{conexion-levantamiento} we deduce that the horizontal lift of a geodesic in $M$ is a geodesic in $\E$. Since not all geodesics are horizontal or vertical, we will need a slight improvement of this argument to classify them all.

\begin{lemma}
All geodesics in $\E$ make a constant angle with a vertical Killing vector field $\xi$.
\end{lemma}

\begin{proof}
Given a geodesic $\gamma$ in $\E$, we can compute
\[\tfrac{d}{dt}\langle\gamma',\xi\rangle=\left\langle\overline\nabla_{\gamma'}\gamma',\xi\right\rangle+\left\langle\gamma',\overline\nabla_{\gamma'}\xi\right\rangle.\]
The first term in the RHS vanishes since $\gamma$ is a geodesic, and the second one also vanishes because $\overline\nabla_{\gamma'}\xi=\tau\gamma'\wedge\xi$ (see Lemma~\ref{lema:curvaturafibrado}).
\end{proof}

Given a real number $\mu\in\R$ and a smooth curve $\alpha:[a,b]\to M$, we can consider the smooth curve
\begin{equation}\label{eqn:levantamiento-angulo-constante}
\gamma:[a,b]\to\E,\quad \gamma(t)=\phi_{\mu t}(\wt\alpha(t)),
\end{equation}
wher $\{\phi_t\}$ is the group of vertical translations associated to a unit vertical vector field $\xi$. The chain rule allows us to compute
 \[\gamma'(t)=\mu\xi_{\gamma(t)}+(d\phi_{\mu t})_{\wt\alpha(t)}(\wt\alpha'(t)),\]
so $\gamma$ makes a constant angle with $\xi$ and will be our candidate to geodesic. Taking into account that $[\wt\alpha',\xi]=0$ and equation~\eqref{conexion-levantamiento}, we get
\begin{equation}\label{eqn:killing-ecuacion-geodesicas}
\overline\nabla_{\gamma'}\gamma'=2\mu\tau\wt\alpha'\wedge\xi+\overline{\nabla_{\alpha'}\alpha'}.\end{equation}
Let us suppose that $\alpha$ has unit speed  and consider $J$ the $\pm\frac{\pi}{2}$-rotation in $TM$ (the sign will be chosen below). Then, there exists a function $\kappa_g:[a,b]\to\R$, the geodesic curvature, such that $\nabla_{\alpha'}{\alpha'}=\kappa_g\cdot J\alpha'$. The horizontal lift of $J\alpha'$ is a horizontal and unitary vector field along $\wt\alpha$, orthogonal to $\wt\alpha'$. Hence, we can choose the sign of $J$  so the horizontal lift of $J\alpha'$ is equal to $-\wt\alpha'\wedge\xi$. Now~\eqref{eqn:killing-ecuacion-geodesicas} implies that $\gamma$ is a geodesic if and only if
\begin{equation}\label{eqn:killing-geodesicas}
\kappa_g(t)=2\mu\,\tau(\alpha(t)).
\end{equation}

\begin{lemma}\label{lema:killing-geodesicas}
Given $\mu\in\R$,  $p\in M$ and $v\in T_pM$, there exists $\varepsilon>0$ and a unique unit-speed smooth curve $\alpha:\ ]-\varepsilon,\varepsilon[\ \to M$ such that $\alpha(0)=p$, $\alpha'(0)=v$ and satisfying~\eqref{eqn:killing-geodesicas}.

Moreover, if $M$ is complete then $\alpha$ extends to the whole real line.
\end{lemma}

\begin{proof}
We will work in a conformal parametrization $\varphi:U\subset\R^2\to M$ compatible with the orientation fixed above, where $U$ is a neighbourhood of $p$. Then, we identify $\alpha$ with the coordinates $(x,y)=\varphi^{-1}\circ\alpha$. Since $\alpha$ has unit speed, there must exist a smooth function $\theta$ such that $x'=\lambda^{-1}\cos\theta$ e $y'=\lambda^{-1}\sin\theta$, where $\lambda$ denotes the conformal factor. The geodesic curvature of $\alpha$ with respect to $J\alpha'=-y'\p_x+x'\p_y$ is given by
\begin{equation*}
\kappa_g=\theta'+\frac{\lambda_x}{\lambda^2}\sin\theta-\frac{\lambda_y}{\lambda^2}\cos\theta.
\end{equation*}
Now, equation~\eqref{eqn:killing-geodesicas} becomes the first-order ODE system:
\begin{equation}\label{eqn:geodesicas-killing}
\left\{\begin{array}{l}
x'=\frac{1}{\lambda(x,y)}\cos\theta,\\
y'=\frac{1}{\lambda(x,y)}\sin\theta,\\
\theta'=2\mu\,\tau(x,y)-\frac{\lambda_x(x,y)}{\lambda(x,y)^2}\sin\theta+\frac{\lambda_y(x,y)}{\lambda^2(x,y)}\cos\theta.
\end{array}\right.
\end{equation}
The general theory of ODEs guarantees the existence of a unique smooth solution in a neighbourhood of the origin when prescribing $\alpha(0)$, $\alpha'(0)$ (note that these initial data are equivalent to $x(0)$, $y(0)$ and $\theta(0)$). Observe that the solution can be extended as long as $\alpha$ is contained in $U$, so if $M$ is complete and we take an atlas consisting of conformal parametrizations compatible with the orientation, then $\alpha$ extends to the whole real line.
\end{proof}

It is important to notice that the curve $\gamma$ given by Lemma~\ref{lema:killing-geodesicas} satisfies $\|\gamma'\|^2=1+\mu^2$ so, after a reparametrization by arc-length, we obtain $\langle\gamma',\xi\rangle=\mu/\sqrt{1+\mu^2}$. This last expression varies in $]-1,1[$ when $\mu\in\R$, so this construction covers all geodesics in $\E$, except for the vertical ones.

\begin{proposition}\label{prop:killing-geodesicas}
Given $p\in\E$, all geodesics in $\E$ passing through $p$ are of one (and only one) of the following types:
\begin{itemize}
 \item[a)] vertical geodesics (fibers of the submersion),
 \item[b)] horizontal lifts of geodesics in $M$ passing through $\pi(p)$,
 \item[c)] of the form $\gamma(t)=\phi_{\mu t}(\wt\alpha(t))$, where $\wt\alpha$ is a horizontal lift of $\alpha$ in $M$ such that $\alpha(0)=\pi(p)$ and satisfying~\eqref{eqn:killing-geodesicas} for some $\mu\neq 0$.
\end{itemize}
In particular, if $M$ is complete, then so is $\E$.
\end{proposition}

\begin{remark}
When the bundle curvature is constant, non-vertical geodesics project into curves with constant geodesic curvature. Moreover, the geodesic is horizontal if and only if its projection is also a geodesic. This gives an easy way to compute geodesics in the $\E(\kappa,\tau)$-spaces.
\end{remark}

\section{Existence results}\label{sec:existence}

When the base is simply-connected, Theorem~\ref{thm:unicidad} gives a uniqueness result for Killing submersions; in this section we will investigate the existence problem and prove that we can prefix any bundle curvature under the same assumption of simple connectivity. 

\subsection{Killing submersions over a disk}
Given an open set $\Omega\subset\R^2$ and $\lambda,\tau\in\mathcal{C}^\infty(\Omega)$, $\lambda>0$, we wonder whether it is possible to solve for $a$ and $b$ in~\eqref{eqn:tau-divergencia-killing}. An explicit way of doing it provided that $\Omega$ is star-shaped is given in the following lemma by just taking $\delta=2\lambda^2\tau$.

\begin{lemma}
Let $\Omega\subset\R$ be open and star-shaped with respect to the origin, and $\delta\in\mathcal{C}^\infty(\Omega)$. Then, $\eta\in\mathcal{C}^\infty(\Omega)$ given by
\[\eta(x,y)=\int_0^1s\ \delta(xs,ys)\df s,\]
satisfies the identity $\delta=(x\eta)_x+(y\eta)_y$.
\end{lemma}

\begin{proof}
It is a direct computation.
\end{proof}

\begin{theorem}\label{teor:clasificacion-killing-local}
Let $\Omega\subseteq\R^2$ be an open set star-shaped with respect to the origin and $\lambda,\tau\in \mathcal{C}^\infty(\Omega)$ with $\lambda>0$. If $\pi:\E\to(\Omega,\lambda^2(\df x^2+\df y^2))$ is a Killing submersion with bundle curvature $\tau$ and $\E$ is simply-connected, then it is isomorphic to the canonical example
\[\pi_1:(\Omega\times\R,\df s^2)\to(\Omega,\lambda^2(\df x^2+\df y^2)),\quad\pi_1(x,y,z)=(x,y),\]
\[\df s^2=\lambda(x,y)^2(\df x^2+\df y^2)+(\df z+\eta(x,y)(y\df x-x\df y))^2,\]
where the function $\eta\in \mathcal{C}^\infty(\Omega)$ is given by
\begin{equation}\label{eqn:eta}
\eta(x,y)=2\int_0^1s\ \tau(xs,ys)\ \lambda(xs,ys)^2\df s.
\end{equation}
\end{theorem}

\begin{remark}
Note that star-shapeness makes everything explicit but an existence and uniqueness theorem also holds in the (more general) simply-connected case. It suffices to conformally parametrize such a simply-connected domain by a disk and apply Theorem~\ref{teor:clasificacion-killing-local}.
 \end{remark}

\begin{remark}
If we drop the condition that $\E$ is simply-connected, it can be easily shown that any Killing submersion $\pi:\E\to\Omega$ is isomorphic to a Riemannian quotient of the Killing submersion constructed in Theorem~\ref{teor:clasificacion-killing-local} by a vertical translation. In particular, $\E$ is diffeomorphic to $\Omega\times\s^1$.
\end{remark}

It is interesting to paraticularize Theorem~\ref{teor:clasificacion-killing-local} to the case $M=\M^2(\kappa)$, the complete simply-connected surface with constant Gaussian curvature $\kappa\in\R$, to get models for all Killing submersions over $\R^2$, $\h^2(\kappa)$ and $\s^2(\kappa)$ minus a point. Given $\kappa\in\R$, we define $\lambda_\kappa\in\mathcal{C}^\infty(\Omega_\kappa)$ as \[\lambda_\kappa(x,y)=\left(1+\tfrac{\kappa}{4}(x^2+y^2)\right)^{-1},\]
where
\[\Omega_k=\begin{cases}\{(x,y)\in\R^2:x^2+y^2<\frac{-4}{\kappa}\}&\text{if }\kappa<0,\\\R^2&\text{if }\kappa\geq 0.\end{cases}\]
Then, the metric $\lambda_\kappa^2(\df x^2+\df y^2)$ in $\Omega_\kappa$ has constant Gaussian curvature $\kappa$. If $\tau$ is constant, then $\eta=\tau\,\lambda_\kappa$ in~\eqref{eqn:eta}, and we obtain the metrics of the spaces $\E(\kappa,\tau)\equiv\Omega_\kappa\times\R$ given by Daniel in~\cite[Section~2.3]{Daniel07}:
\[\lambda_\kappa^2(\df x^2+\df y^2)+(\df z+\tau\,\lambda_\kappa\,(y\df x-x\df y))^2.\]
Recall that we are not considering a whole fiber of a point in $\s^2(\kappa)$ for $\kappa>0$. The global case will be treated in the next section.

\subsection{Killing submersions over a $2$-sphere}

We can define Killing submersions over $\s^2$ so different as the Riemannian products $\s^2\times\R$ and $\s^2\times\s^1$ (both with $\tau=0$) or the Berger spheres and the lens spaces $L(n,1)$ via the Hopf projection (see Remark~\ref{rmk:hopf} below). Along this section, we will suppose that the surface playing the role of base surface is $(\s^2,g)$ for some Riemannian metric $g$. 

Unlike in the cases treated above, this surface is compact. Hence, given a Killing submersion $\pi:\E\to(\s^2,g)$ and its bundle curvature $\tau\in \mathcal{C}^\infty(M)$, the \emph{total bundle curvature} 
\[T=\int_M\tau\]
is well-defined and finite. This quantity will make the difference between the possible topologies of the total space.

\subsubsection{The case $T=0$}

\begin{proposition}\label{prop:s2xR}
Let $\pi:\E\to(\s^2,g)$ be a Killing submersion with total bundle curvature $T=0$. Then, the submersion admits a global smooth section.
\begin{itemize}
 \item[a)] If the length of the fibers of $\pi$ is infinite, then it is isomorphic to
\[\pi_1:(\s^2\times\R,\df s^2)\to(\s^2,g),\quad\pi_1(p,t)=p,\]
for some Riemannian metric $\df s^2$ defined in $\s^2\times\R$ and such that $\p_t$ is a unit vertical Killing vector field.
 \item[b)] Otherwise, the Killing submersion is isomorphic to the Riemannian quotient of the example in item (a) by some vertical translation.
\end{itemize}
\end{proposition}

\begin{proof}
The condition $T=0$ guarantees the existence of a great circle $\Gamma\subset\s^2$ such that $D_1$ and $D_2$, the two open components of $\s^2\sm\Gamma$, satisfy \[\textstyle\int_{D_1}\tau=\int_{D_2}\tau=0.\]
Let $\wt\Gamma\subset\E$ be any horizontal lift of $\Gamma$. If the fibers of $\pi$ have infinite length, then Proposition~\ref{prop:holonomia} implies that $\wt\Gamma$ is a closed curve in $\E$. For $i\in\{1,2\}$, as $\wt\Gamma$ lies in the boundary of $\pi^{-1}(\overline D_i)$ and projects one-to-one by $\pi$ onto $\Gamma$, there exists a section $F_i:\overline{D_i}\to\E$ with $F_i(\Gamma)=\wt\Gamma$. Thus $F:\s^2\to\E$ defined by $F=F_i$ in $\overline{D}_i$ is a global continuous section, and there is no loss of generality in supposing that  $F$ is smooth (just by perturbing it in a neighborhood of $\Gamma$). Then $\Psi:\s^2\times\R\to\E$ given by $\Psi(p,t)=\phi_t(F(p))$ is a global diffeomorphism, where $\phi_t$ denotes the $1$-parameter group of vertical isometries. The induced metric $\df s^2$ in $\s^2\times\R$ through $\Psi$ satisfies the requirements of item (a). 

In the case that the length of the fibers of $\pi$ is finite, we can work in the universal cover of $\pi^{-1}(\overline{D_i})$, for $i\in\{1,2\}$, repeat the arguments above, and finally take a convenient quotient by a vertical translation.
\end{proof}

The rest of this section is devoted to obtain explicit models for the metrics in $\s^2\times\R$ making the projection over the first factor a Killing submersion. This will show that Proposition~\ref{prop:s2xR} is sharp, but we will also obtain a quite general method for constructing Killing submersions.

Inspired by the canonical metrics in~\eqref{eqn:metrica-general-killing}, let us consider arbitrary functions $a_1,\ldots,a_n\in\mathcal{C}^\infty(\R^n)$ and the projection over the first $n$ coordinates $\pi:(\R^{n+1},\df s^2)\to\R^n$. Here, we endow $\R^n$ with the usual metric and 
\begin{equation}\label{eqn:metrica-killing-Rn}
\textstyle\df s^2=\sum_{k=1}^n\df x_i^2+\left(\df t-\sum_{k=1}^na_k\df x_k\right)^2,
\end{equation}
where we denote by $(x_1,\ldots,x_n,t)$ the usual coordinates of $\R^{n+1}$. Then, $\pi$ is a Riemannian submersion whose fibers are the integral curves of the unit Killing vector field $\p_t$ in $(\R^{n+1},\df s^2)$. 

Given a smooth orientable surface $\Sigma$, we can isometrically embed it in $\R^n$ for some $n\in\N$ by the Nash embedding Theorem~\cite{Nash}. Then, we shall consider the metric induced by~\eqref{eqn:metrica-killing-Rn} in $\Sigma\times\R\subset\R^{n+1}$. Obviously, $\pi$ restricts to a Killing submersion $\Sigma\times\R\to\Sigma$. We will now compute its bundle curvature in terms of the functions $a_k$, but we will first need a convention for the orientation in $\Sigma\times\R$: if a local frame $\{e_1,e_2\}$ in $\Sigma$ is positively oriented, then $\{E_1,E_2,\p_t\}$ will be said positively oriented in $\Sigma\times\R$, where $E_i$ is the horizontal lift of $e_i$ for $i\in\{1,2\}$.

\begin{proposition}\label{prop:ejemplos_s2xR}
Let $\Sigma$ be a smooth oriented surface isometrically embedded in $\R^n$. The Killing submersion $\Sigma\times\R\to\Sigma$ defined above has bundle curvature
\[\tau=\tfrac{1}{2}\,\div_\Sigma(JT),\]
where $T=(\p_t)^\top\in\X(\Sigma)$ is the component of $\p_t$ tangent to $\Sigma\equiv\Sigma\times\{0\}\subset\R^{n+1}$ with respect to $\df s^2$, and $J:\X(\Sigma)\to\X(\Sigma)$ is a $\frac\pi2$-rotation in $T\Sigma$.
\end{proposition}

\begin{proof}
Let $X:\Omega\subset\R^2\to\Sigma$ be a local conformal parametrization of $\Sigma$ with conformal factor $\lambda\in\mathcal{C}^\infty(\Omega)$, and such that $\{\frac{1}{\lambda}X_u,\frac{1}{\lambda}X_v\}$ is a positively oriented orthonormal frame of $T\Sigma$. Let $\{E_1,E_2\}\subset\X(X(\Omega)\times\R)$ be a horizontal lift of the frame $\{\frac{1}{\lambda}X_u,\frac{1}{\lambda}X_v\}$ which, together with $E_3=\partial_t$, is a positively oriented orthonormal frame in $X(\Omega)\times\R$. As in equation~\eqref{eqn:tau-divergencia-killing}, we can compute $\tau$ from the identity $2\tau=\langle [E_1,E_2],E_3\rangle$. Note that there exist $f,g\in\mathcal{C}^\infty(X(\Omega))$ such that $E_1=\frac{1}{\lambda}X_u+f\p_t$ and $E_2=\frac{1}{\lambda}X_v+g\p_t$, so
\begin{align*}
 [E_1,E_2]&=[\tfrac{1}{\lambda}X_u,\tfrac{1}{\lambda}X_v]+[\tfrac{1}{\lambda}X_u,g\p_t]+[f\p_t,\tfrac{1}{\lambda}X_u]+[f\p_t,g\p_t]\\
&=\tfrac{1}{\lambda^3}(\lambda_vX_u-\lambda_uX_v)+\tfrac{1}{\lambda}(g_u-f_v)\p_t.
\end{align*}
Moreover, since $0=\langle E_1,E_3\rangle=\langle\frac{1}{\lambda}X_u+f\p_t,\p_t\rangle$, we deduce that $\langle X_u,\p_t\rangle=-\lambda f$ and, analogously, $\langle X_v,\p_t\rangle=-\lambda g$. Hence,
\[2\tau=\langle [E_1,E_2],E_3\rangle=\tfrac{1}{\lambda^2}\left((\lambda g)_u-(\lambda f)_v\right)=\div_\Sigma(Y),\]
where $Y\in\X(\Sigma)$ is the vector field $\frac{g}{\lambda}X_u-\frac{f}{\lambda}X_v$. From here, it is easy to check that $Y=JT$ and we are done.
\end{proof}

\begin{remark}\label{rmk:compact-divergence}
If $\Sigma$ is compact, then $\int_\Sigma\tau=\frac{1}{2}\int_\Sigma\div(JT)=0$ as an application of the divergence theorem. Conversely, every function on a compact orientable surface $\Sigma$ with zero integral is well-known to be the divergence of some vector field on $\Sigma$.
\end{remark}

As a particular case, we may consider the round sphere
\[\s^2(\kappa)=\left\{(x,y,z)\in\R^3:x^2+y^2+z^2=\tfrac{1}{\kappa}\right\}\subset\R^3.\]
and endow $\s^2\times\R\subset\R^4$ with the metric given by~\eqref{eqn:metrica-killing-Rn} for $n=3$ and some $a_1,a_2,a_3\in \mathcal{C}^\infty(\R^3)$. The stereographic projection $X:\R^2\to\s^2(\kappa)\sm\{(0,0,1/\sqrt{\kappa})\}$ defined by
\begin{equation}\label{eqn:stereographic}
X(u,v)=\left(\frac{2u}{\kappa(u^2+v^2)+1},\frac{2v}{\kappa(u^2+v^2)+1},\frac{1}{\sqrt{\kappa}}\frac{\kappa(u^2+v^2)-1}{\kappa(u^2+v^2)+1}\right),
\end{equation}
allows us to work out the bundle curvature $\tau$ of the induced Killing submersion $\s^2(\kappa)\times\R\to\s^2(\kappa)$ as in the proof of Proposition~\ref{prop:ejemplos_s2xR}. We get
\[2\tau=\sqrt{\kappa}\left((ya_3-za_2)_x+(za_1-xa_3)_y+(xa_2-ya_1)_z\right).\]

\subsubsection{The case $T\neq0$} 
Let us consider the $3$-sphere 
\[\s^3=\{(z,w)\in\C^2:|z|^2+|w|^2=1\}\subset\C^2\]
and $\s^2(\kappa)=\{(z,t):|z|^2+t^2=\frac{1}{\kappa}\}\subset\C\times\R$ for $\kappa>0$. The submersion
\begin{equation}\label{eqn:fibracion-hopf}
 \begin{array}{rcl}
  \pi_{\text{Hopf}}:\s^3&\longrightarrow&\s^2(\kappa)\\
  (z,w)&\longmapsto&\frac{1}{\sqrt{\kappa}}(2z\bar w,|z|^2-|w|^2)
 \end{array}
\end{equation}
is known as the \emph{Hopf projection}. The fiber passing through $(z,w)\in\s^3$ is given by $\{(e^{it}z,e^{it}w):t\in\R\}$ and the $1$-parameter group of diffeomorphisms $\phi_t(z,w)=(e^{it}z,e^{it}w)$, $t\in\R$, satisfies that the orbit of a point by the action of the group coincides with its fiber with respect to the submersion.

\begin{remark}\label{rmk:hopf}
Given a natural number $n\in\N$, we can consider the quotient of $\s^3$ under the group of diffeomorphisms $G_n=\{\phi_{2\pi k/n}:k\in\{1,\ldots,n\}\}$, which is cyclic and has order $n$. The quotient $\s^3/G_n$ is known as the lens space $L(n,1)$. The condition $\pi_{\text{Hopf}}\circ\phi_t=\pi_{\text{Hopf}}$ guarantees that $\pi_{\text{Hopf}}$ induces a submersion $\pi_n:L(n,1)\to\s^2(\kappa)$. Observe that, for any $n\in\N$, $L(n,1)$ is orientable and its fundamental group is isomorphic to the cyclic group of order $n$, so two lens spaces $L(n,1)$ and $L(m,1)$ are not homeomorphic for $m\neq n$ (see~\cite{sav12} for a more detailed description).

If we endow $\s^3$ with a metric making $\pi_{\text{Hopf}}$ a Killing submersion, then the fibers of $\pi_{\text{Hopf}}$ have finite length (they are compact) and it is easy to check that $\pi_n$ is a Killing submersion when we consider the quotient metric, for all $n$. Moreover, the length of the fibers of $\pi_{\text{Hopf}}$ in $\s^3$ is $n$ times the length of the corresponding fibers of $\pi_n$ in $L(n,1)$.
\end{remark}

\begin{proposition}\label{prop:s3}
Let $\pi:\E\to(\s^2,g)$ be a Killing submersion with total bundle curvature $T\neq0$. Then, there exists $n\in\N$ such that the length of the fibers is equal to $|2T|/n$.
\begin{itemize}
 \item[a)] If $n=1$, then $\pi:\E\to(\s^2,g)$ is isomorphic to the Hopf fibration
\[\pi_{\mathrm{Hopf}}:(\s^3,\df s^2)\to(\s^2,g),\ \pi_{\mathrm{Hopf}}(z,w)=(2z\bar w,|z|^2-|w|^2),\]
for some Riemannian metric $\df s^2$ in $\s^3$ such that $\xi_{(z,w)}=\frac{\pi}{T}(iz,iw)$ is a unit Killing vector field.
 \item[b)] If $n>1$, then $\pi:\E\to(\s^2,g)$ is isomorphic to the Riemannian quotient of a submersion as in item (a) by a vertical translation of length  $|2T|/n$.
\end{itemize}
\end{proposition}

\begin{proof}
As in the proof of Proposition~\ref{prop:s2xR}, let us take a geodesic $\Gamma$ which divides $\s^2$ in two hemispheres $D_1$ and $D_2$ such that \[\textstyle\int_{D_1}\tau=\int_{D_2}\tau=\frac{T}{2}.\]
We parametrize $\Gamma$ as $\gamma:[a,b]\to\s^2$ and a horizontal lift $\wt\Gamma$ of $\Gamma$ as $\wt\gamma:[a,b]\to\E$. The universal Riemannian covering space of $\pi^{-1}(\overline{D}_i)$, $i\in\{1,2\}$, will be denoted by $W_i\equiv\overline D_i\times\R$, and is a closed solid cylinder. The curve $\wt\Gamma$ can be lifted to both $W_1$ and $W_2$. Since the outer conormal vector fields to $\overline{D}_1$ and $\overline{D}_2$ along their boundary have opposite directions, the difference of heights between $\wt\gamma(a)$ and $\wt\gamma(b)$ when we consider them in $W_1$ or $W_2$ is equal to $|T|$, but they have opposite signs (see the proof of Proposition~\ref{prop:holonomia}). In other words,  we will arrive to $\wt\gamma(b)$ after traveling vertically from $\wt\gamma(a)$ a distance of $|T|$, and, if we continue from $\wt\gamma(b)$, we will arrive again to $\wt\gamma(a)$ after the same distance (see Figure~\ref{fig:construccion-hopf}). Thus, the length of the fibers is an integer divisor of $|2T|$. In particular, $\pi^{-1}(\overline{D}_1)$ and $\pi^{-1}(\overline{D}_2)$ are solid tori.

\begin{figure}
\begin{center}
\includegraphics[height=4cm]{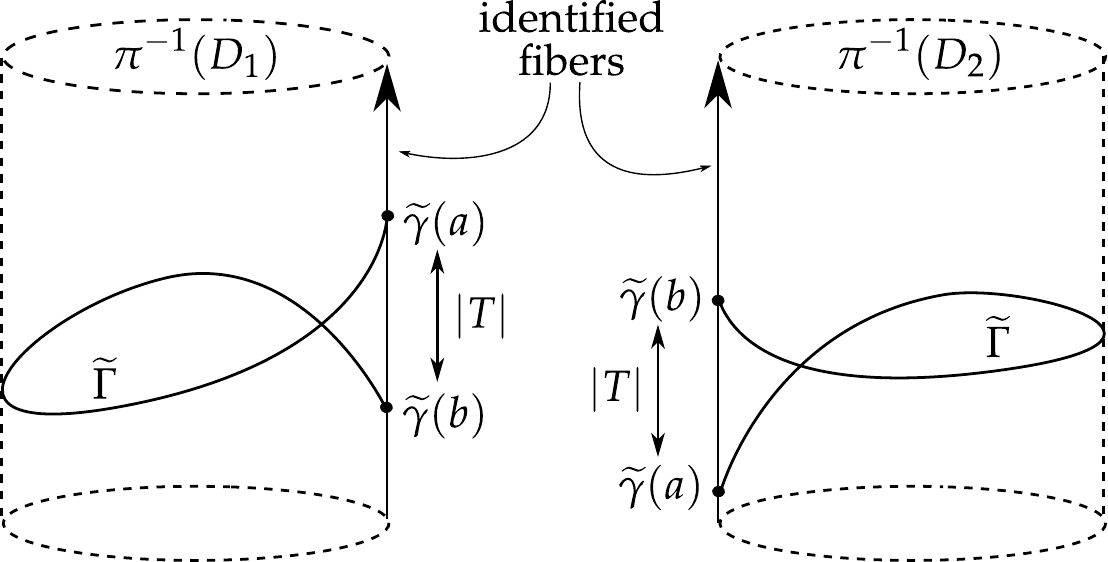}
\end{center}
\caption{The curve $\wt\Gamma$ is represented in the solid cylinders $W_1$ and $W_2$ covering $\pi^{-1}(D_1)$ and $\pi^{-1}(D_2)$, respectively, and its endpoints lie on the vertical geodesic containing the vertical arrow representing a global vertical Killing vector field. After gluing along this geodesic, we conclude that the length of the fibers is an integer divisor of $|2T|$.}\label{fig:construccion-hopf}
\end{figure}

Now, observe that the curve $\wt\Gamma$ determines how $\pi^{-1}(\overline{D_1})$ and $\pi^{-1}(\overline{D_2})$ must be glued together, and $\wt\Gamma$ turns $n$ times in the vertical direction, so we can work in a $n$-sheet vertical covering space of both tori, where $\wt\Gamma$ will look like as in Figure~\ref{fig:construccion-hopf} after identifying the top and bottom faces of the cylinders. This way of gluing the two tori along $\Gamma$ provides a manifold diffeomorphic to $\s^3$, and the induced fibration is the Hopf fibration (see~\cite{sav12}). By pulling the metric in $\E$ back via this diffeomorphism, item (a) in the statement follows. Item (b) is also proved since we only need to undo the covering space procedure by taking a quotient with respect to a vertical translation of length $|2T|/n$.
\end{proof}

We can now combine the local existence given by Theorem~\ref{teor:clasificacion-killing-local} with Propositions~\ref{prop:s2xR} and~\ref{prop:s3} to obtain a description of \emph{all} Killing submersions over a Riemannian $2$-sphere.

\begin{theorem}\label{thm:existence-S2}
Let $g$ be a Riemannian metric on $\s^2$ and $\tau\in\mathcal{C}^\infty(\s^2)$. Up to isomorphism, there exists a unique Killing submersion over $(\s^2,g)$ with bundle curvature $\tau$ and whose total space is simply-connected. 
\end{theorem}

\begin{proof}
The uniqueness is a consequence of Theorem~\ref{thm:unicidad} and the description of the length of the fibers in Propositions~\ref{prop:s2xR} and~\ref{prop:s3}. We will now assume that $T=\int_{(\s^2,g)}\tau\neq 0$ (the case $T=0$ is similar) and prove its existence.

 Consider a great circle $\Gamma\subset\s^2$ splitting $\s^2$ in two hemispheres $D_1$ and $D_2$. By applying Theorem~\ref{teor:clasificacion-killing-local} in a neighborhood of $\overline D_1$ and $\overline D_2$, we obtain Killing submersions $\pi_1$ and $\pi_2$ over such neighborhoods with the desired bundle curvature and non-compact fibers. The argument in proposition~\ref{prop:s3} guarantees that, after taking the quotient by vertical translations of length $|2T|$, the two submersions can be glued together along $\pi^{-1}(\Gamma)$ to produce a (continuous) submersion $\pi:\s^3\to\s^2$. In order to prove that $\pi$ is smooth along $\pi^{-1}(\Gamma)$, observe that both $\pi_1$ and $\pi_2$ are defined in a neighborhood of $\Gamma$ where they share the same bundle curvature. Thus they locally coincide by Theorem~\ref{thm:unicidad} in a neighborhood of each $p\in\pi ^{-1}(\Gamma)$.
\end{proof}

In the previous section, we showed a constructive method to obtain trivial Killing submersions in a global way. Now, we will do the same for Killing submersions with $T\neq 0$ for round spheres $\s^2(\kappa)$ as base surfaces, though the method can be also adapted to the case $T=0$.

Let us consider the Hopf fibration given by~\eqref{eqn:fibracion-hopf} and the global frame in $\s^3\subset\C^2$ defined by
\begin{align*}
 (E_1)_{(z,w)}&=(-\bar w,\bar z),& (E_2)_{(z,w)}&=(-i\bar w,i\bar z),& (E_3)_{(z,w)}&=(iz,iw).
\end{align*}
This frame is orthonormal when we endow $\s^3$ with the round metric of curvature one. Let $\tau\in\mathcal{C}^\infty(\s^2(\kappa))$ be a function with integral $T\neq 0$. Note that $\tau$ induces a function in $\wt\tau\in \mathcal{C}^\infty(\R^2)$ via the stereographic projection given by~\eqref{eqn:stereographic}. Theorem~\ref{teor:clasificacion-killing-local} allows us to construct a Killing submersion over $\s^2(\kappa)\sm\{(0,0,1/\sqrt{\kappa})\}$ with bundle curvature $\wt\tau$. To do this, we calculate the associated function $\wt\eta\in\mathcal{C}^\infty(\R^2)$ given by
\[\wt\eta(x,y)=2\int_0^1\frac{s\cdot\wt\tau(sx,sy)}{\left(1+\frac{\kappa}{4}s^2(x^2+y^2)\right)^2}\df s,\]
which extends smoothly to infinity since $\wt\tau$ extends smoothly to infinity, and thus induces $\eta\in\mathcal{C}^\infty(\s^2(\kappa))$ by pulling back via the stereographic projection again. Hence this construction induces a Riemannian metric in $\s^3$ minus the fiber of $(0,0,1/\sqrt{\kappa})$ but can be extended to the whole $\s^3$. It can be shown that this metric in $\s^3$ is the determined by the fact that
\begin{align*}
 Y_1&=\frac{\sqrt{\kappa}}{2}E_1-\frac{\im(zw)(\kappa T|w|^2-4\pi\eta(\pi_{\text{Hopf}}(z,w)))}{2\pi\sqrt{\kappa}|w|^4}E_3,\\
 Y_2&=\frac{\sqrt{\kappa}}{2}E_2+\frac{\im(zw)(\kappa T|w|^2-4\pi\eta(\pi_{\text{Hopf}}(z,w)))}{2\pi\sqrt{\kappa}|w|^4}E_3,&
 Y_3&=\frac{\pi}{T}E_3,
\end{align*}
defines a global orthonormal frame. If $\tau$ is constant, then $\kappa T=4\pi\tau$ and $\eta(\pi_{\text{Hopf}}(z,w))=|w|^2\tau$ so the coefficients of $E_3$ in $Y_1$ and $Y_2$ vanish, and we get the metrics of the Berger spheres given by Torralbo~\cite{Tor12}.

\section{Characterization of homogeneous Killing submersions}\label{sec:homogeneos}

Recall that a Riemannian manifold is said \emph{homogeneous} when its isometry group acts transitively on the manifold. In this section, we will characterize the $\E(\kappa,\tau)$-spaces as the only simply-connected homogeneous $3$-manifolds admitting the structure of a Killing submersion.

In order to obtain this result, we will compute the Riemannian curvature of the total space $\E$ of a Killing submersion $\pi:\E\to M$ in terms of $M$ and the bundle curvature $\tau$. Since the computation is purely local, we will work in a canonical example (see Definition~\ref{defi:canonical}) associated to some functions $\lambda,a,b\in \mathcal{C}^\infty(\Omega)$ with $\lambda>0$ and $\Omega\subset\R^2$ (a different approach can be found in~\cite{EO}). Koszul formula yields the Levi-Civita connection in the canonical orthonormal frame $\{E_1,E_2,E_3\}$ given by~\eqref{eqn:base-universal-killing}:
\begin{equation}\label{eqn:killing-conexion}
\begin{array}{lll}
\overline\nabla_{E_1}E_1=-\frac{\lambda_y}{\lambda^2}E_2,&\overline\nabla_{E_1}E_2=\frac{\lambda_y}{\lambda^2}E_1+\tau E_3,&\overline\nabla_{E_1}E_3=-\tau E_2,\\
\overline\nabla_{E_2}E_1=\frac{\lambda_x}{\lambda^2}E_2-\tau E_3,&\overline\nabla_{E_2}E_2=-\frac{\lambda_x}{\lambda^2}E_1,&\overline\nabla_{E_2}E_3=\tau E_1,\\
\overline\nabla_{E_3}E_1=-\tau E_2,&\overline\nabla_{E_3}E_2=\tau E_1,&\overline\nabla_{E_3}E_3=0.
\end{array}
\end{equation}
Since the Gaussian curvature $K_M$ of $M$ can be written in terms of the conformal factor as
\[K_M=-\frac{\Delta_0(\log\lambda)}{\lambda^2}=\frac{\lambda_x^2+\lambda_y^2}{\lambda^4}-\frac{\lambda_{xx}+\lambda_{yy}}{\lambda^3},\]
it is easy to work out any sectional curvature in $\E$.

\begin{lemma}\label{lema:curvatura-seccional}
Let $\pi:\E\to M$ be a Killing submersion and $p\in\E$. Given a linear plane $\Pi\subseteq T_p\E$ with normal vector  $N\in T_p\E$, its sectional curvature is
\[K(\Pi)=\nu^2(K_M-3\tau^2)+(1-\nu^2)\tau^2-2\nu\langle N\wedge\xi_p,(\overline\nabla\tau)_p\rangle,\]
where $\nu=\langle N,\xi_p\rangle$, $\xi$ denotes the unit Killing vector field, $K_M$ is the Gaussian curvature of $M$ at $\pi(p)$, and $\tau$ is the bundle curvature at $p$.
\end{lemma}

The sectional curvature is $K_M-3\tau^2$ for horizontal planes (i.e., planes which are orthogonal to $\xi$) and $\tau^2$ for vertical planes (i.e., planes containing the direction $\xi$). In particular, we deduce that $\h^3$, the hyperbolic $3$-space does not admit a Killing submersion structure since $\h^3$ has constant sectional curvature of $-1$ and vertical planes in a Killing submersion always have non-negative sectional curvature.

On the other hand, given $v\in T_p\E$ with $\|v\|=1$, the Ricci curvature of $v$ can be easily deduced from Lemma~\ref{lema:curvatura-seccional} as
\begin{equation}\label{eqn:ricci1}
\Ric(v)=(K_M-2\tau^2)-\langle v,\xi_p\rangle^2(K_M-4\tau^2)+2\langle v,\xi_p\rangle\langle v\wedge\xi_p,\overline\nabla\tau\rangle.
\end{equation}
The scalar curvature is $\rho=2(K_M-\tau^2)$.

\begin{theorem}\label{thm:killing-homogeneos}
Let $\pi:\E\to M$ be a Killing submersion. If $\E$ is homogeneous, then both the Gaussian curvature of $M$ and the bundle curvature are constant. In particular, $\E$ is a $\E(\kappa,\tau)$-space or its quotient by a vertical translation.
\end{theorem}

\begin{proof}
Given $p\in\E$ and $v\in T_p\E$ with $\|v\|=1$, we can decompose $v=u+\sigma\xi_p$, where $u$ is horizontal and $\sigma\in\R$. From~\eqref{eqn:ricci1}, we get
\begin{equation*}\label{eqn:ricci2}
\Ric(v)=(K_M-2\tau^2)+\sigma\cdot\bigl(\langle u\wedge\xi_p,(\overline\nabla\tau)_p\rangle-(K_M-4\tau^2)\sigma\bigr).
\end{equation*}
Let $U_p=\{v\in T_p\E:\|v\|=1\}$ and $A_p=\{v\in U_p:\Ric(v)=K_M-2\tau^2\}$. Observe that the vectors $v\in U_p$ satisfying $\sigma=0$ form a great circle and the same happens for $\langle u\wedge\xi_p,(\overline\nabla\tau)_p\rangle-(K_M-4\tau^2)\sigma=0$ if $(\overline\nabla\tau)_p\neq 0$ or $K_M\neq4\tau^2$. We deduce that
\begin{equation}\label{eqn:Ap}
A_p=\begin{cases}
 U_p,&\text{if }K_M=4\tau^2\text{ and }(\overline\nabla\tau)_p=0,\\
 \text{a great circle},&\text{if }K_M\neq4\tau^2\text{ and }(\overline\nabla\tau)_p=0,\\
 \text{two great circles},&\text{if }(\overline\nabla\tau)_p\neq0.
\end{cases}
\end{equation}

Let $f:\E\to\E$ be an isometry. Since any two great circles in a sphere intersect and $\df f_p$ maps great circles in $U_p$ to great circles in $U_{f(p)}$, we deduce that $\df f_p(A_p)\cap A_{f(p)}\neq\emptyset$. As a consequence, $K_M-2\tau^2$ attains the same value at the points $p$ and $f(p)$. If $\E$ is homogeneous, then this implies that $K_M-2\tau^2$ is constant but, on the other hand, the scalar curvature $2(K_M-\tau^2)$ is also constant, from where both $K_M$ and $\tau$ are constant.
\end{proof}

\begin{remark}
Given a $3$-dimensional metric Lie group $G$ (i.e., it is endowed with a left-invariant metric) with isometry group of dimension $3$, it is homogeneous. We deduce that the set of points where a Killing vector field (i.e., a right-invariant vector field) is unitary has empty interior. Otherwise, this open subset would be locally isometric to a $\E(\kappa,\tau)$-space, and this is impossible (see~\cite{mp11} for a detailed description of metric Lie groups).
\end{remark}

Finally, let us mention that the condition $K_M=4\tau^2$ does not imply that $\E$ has constant sectional curvature (unless $\tau$ is constant), but it says that horizontal and vertical planes have the same sectional curvature. Note that, if $(\nabla\tau)_p\neq 0$ and $K_M=4\tau^2$ at some $p\in M$, then the set $A_p$ in~\eqref{eqn:Ap} consists of two \emph{orthogonal} great circles in the unit sphere $U_p$.

\end{document}